%% file: SummarypaperFinal.tex
\documentclass[a4]{article}
\usepackage[pdftex]{graphicx,color}
\usepackage{amssymb}
\usepackage{amsmath,amsthm,bbm,mathrsfs}
\usepackage[italian,francais,german,english]{babel}

\numberwithin{equation}{section}
\newcommand{\beq}{\begin{equation}}
\newcommand{\norm}[1]{\lVert #1 \rVert}
\newcommand{\eeq}{\end{equation}}
\newcommand{\beqn}{\begin{eqnarray}}
\newcommand{\eeqn}{\end{eqnarray}}
\newcommand{\bal}{\begin{align}}
\newcommand{\eal}{\end{align}}
\renewcommand{\Re}{\mathrm{Re}\,}

\newcommand{\scalar}[2]{\langle{#1} \mspace{2mu}, {#2}\rangle}

\def\({\left(}
\def\){\right)}

\def\11{{\mathbbmss 1}}

\def\i{{\mathrm i}}
\def\e{{\mathrm e}}

\def\mez{{-\frac{1}{2}}}
\def\pez{{\frac{1}{2}}}
\def\mdz{{-\frac{3}{2}}}
\def\pdz{{\frac{3}{2}}}

\newtheorem{theorem}{Theorem}[section]
\newtheorem{proposition}[theorem]{Proposition}
\newtheorem{lemma}[theorem]{Lemma}

\newtheorem*{rem*}{Remarks}
\begin{document}

\title{Some Hamiltonian Models of Friction}
\author{J\"urg Fr\"ohlich\footnote{juerg@itp.phys.ethz.ch},\  Zhou Gang\footnote{zhougang@itp.phys.ethz.ch} and Avy Soffer\footnote{soffer@math.rutgers.edu}}
\maketitle
\setlength{\leftmargin}{.1in}
\setlength{\rightmargin}{.1in}
\normalsize \vskip.1in
\setcounter{page}{1} \setlength{\leftmargin}{.1in}
\setlength{\rightmargin}{.1in}
\large
\centerline{$^{\ast,\dagger}$Institute for Theoretical Physics, ETH Zurich, CH-8093, Z\"urich, Switzerland}
\centerline{$^{\ddagger}$Department of Mathematics, Rutgers University, New Jersey 08854}
\date
\abstract{\large Mathematical results on some models describing the motion of a tracer particle through a Bose-Einstein condensate are described. In the limit of a very dense, very weakly interacting Bose gas and for a very large particle mass, the dynamics of the coupled system is determined by classical non-linear Hamiltonian equations of motion. The particle's motion exhibits deceleration corresponding to friction (with memory) caused by the emission of Cerenkov radiation of gapless modes into the gas.

Precise results are stated and outlines of proofs are presented. Some technical details are deferred to forthcoming papers.}
\large

\section{Background from Physics}
In this introductory section, we recall some results of an analysis presented in a companion paper \cite{FSSZ}.

The physical system studied in this paper consists of a heavy, non-relativistic tracer particle interacting with the non-relativistic atoms in a Bose gas. Our purpose is to introduce some mathematical models describing this system and to analyze their properties in the limiting regime where the density of the Bose gas becomes very large, the interactions between atoms in the gas become very weak, and the mass of the tracer particle is very large. This regime is commonly called ``{\it{mean-field limit}}". In this limit, the dynamics of the system is determined by classical, nonlinear {\it{Hamiltonian equations of motion.}} If the gas exhibits Bose-Einstein condensation these equations describe a process of emission of gapless modes into the gas by the tracer particle analogous to the {\it{Cerenkov radiation}} emitted by a fast charged particle moving through an optically dense medium, \cite{JDJ, KM, FSSZ}. The emission of Cerenkov radiation causes a deceleration of the particle's motion. Our models thus describe a Hamiltonian mechanism of friction with memory. Our aim is to describe mathematical results on the simplest model exhibiting this mechanism.

We start by considering a gas of $n$ bosonic atoms confined to some cubical region, $\Lambda$, of physical space $\mathbb{R}^{3}.$ The pure states of this gas are unit rays in the Hilbert space
\begin{equation}
\mathcal{H}_{G}^{(n)}=\mathcal{L}^{2}(\Lambda, d^{3}x)^{\otimes_{s}n}
\end{equation} where $\otimes_{s}$ denotes the symmetric tensor product expressing our assumption that the atoms obey Bose-Einstein statistics. The Hamiltonian of the gas is given by the operator
\begin{equation}
H_{G}^{(n)}:=-\sum_{k=1}^{n} \frac{1}{2m}\Delta_{x_k}^{\Lambda}+\lambda \sum_{1\leq k\leq l \leq n}\phi(x_k-x_l),
\end{equation} where $m$ is the mass of an atom, $\Delta_{x}^{\Lambda}$ denotes the Laplacian on $\mathcal{L}^{2}(\Lambda, d^3 x)$ with, e.g., periodic boundary conditions at $\partial \Lambda,$ $\lambda\geq 0$ is a coupling constant, and $\phi$ is a bounded two-body potential of positive type and of fast decay at infinity.

The state space of the tracer particle is given by
\begin{equation}
\mathcal{H}_{p}:=\mathcal{L}^{2}(\Lambda, d^3 x)
\end{equation} and its Hamiltonian is chosen to be
\begin{equation}
H_{p}:=-\frac{1}{2M} \Delta_{X}^{(\Lambda)}+V(X),
\end{equation} where $M$ is the mass and $X\in \Lambda$ the position of the particle, and $V(X)$ is the potential of an external force. The interaction of the particle with the gas is given by
\begin{equation}
H_{I}:=g\sum_{k=1}^{n}W(x_k-X),
\end{equation} where $W$ is a bounded two-body potential of fast decay at infinity, and $g$ is a coupling constant.

The Hilbert space of the total system is
\begin{equation}
\mathcal{H}^{(n)}=\mathcal{H}_{p}\otimes \mathcal{H}_{G}^{(n)},
\end{equation} and the total Hamiltonian is the operator
\begin{equation}
\tilde{H}^{(n)}=H_{p}+H_{G}^{(n)}+H_{I},
\end{equation} where $H_{p}$ stands for $H_{p}\otimes I|_{\mathcal{H}_{G}^{(n)}}$ and $H_{G}^{(n)}$ for $I|_{\mathcal{H}_{p}}\otimes H_{G}^{(n)}.$

We will be interested in studying the properties of this system in the thermodynamic limit
\begin{equation}
\Lambda \nearrow \mathbb{R}^{3},\ \text{with}\ \rho:=\frac{n}{|\Lambda|}\ \text{kept fixed},
\end{equation} where $|\Lambda|$ denotes the volume of $\Lambda$ and $\rho$ is the particle density of the gas. For this purpose, it is convenient to describe our model in the formalism of second quantization.

The bosonic Fock space of the Bose gas is the Hilbert space
\begin{equation}
\mathcal{F}_{G}:=\oplus_{n=0}^{\infty}\mathcal{H}_{G}^{(n)},
\end{equation} the total Hilbert space is given by
\begin{equation}\label{eq:Hamiltonian1}
\mathcal{H}=\mathcal{H}_{p}\otimes \mathcal{F}_{G},
\end{equation} and the Hamiltonian by
\begin{equation}
\tilde{H}=\oplus_{n=0}^{\infty} \tilde{H}^{n}|_{\mathcal{H}^{(n)}}.
\end{equation} In order to tune the density of the Bose gas in states of small total energy to a desired value $\approx \rho$, we replace the operator $\tilde{H}$ by
\begin{equation}
H:=\oplus_{n=0}^{\infty}(\tilde{H}^{(n)}-\mu n)|_{\mathcal{H}^{n}}+\text{constant}.
\end{equation}
Note that, on states of a fixed number, $n$, of atoms, $H$ differs from $\tilde{H}$ just by the constant - $\mu n+\text{constant},$ and hence describes the same physics. Setting $\mu:=\lambda \hat{\phi}(0)\rho-\frac{\lambda}{2} \phi(0)$, $\text{constant}:=-g \hat{W}(0) \rho+\frac{\lambda}{2} \hat{\phi}(0) \rho^2|\Lambda|,$ where $\ \hat{}\  $ indicates Fourier transformation, we find for $H$, in the formalism of second quantization, the expression
\begin{equation}\label{eq:Hamiltonian}
\begin{array}{lll}
H&=&-\frac{1}{2M}\Delta_{X}+V(X)+\int \frac{1}{2m} (\nabla a^*)(x) (\nabla a)(x)\ d^3 x\\
& &\\
& &+g\int W(x-X) (a^*(x) a(x)-\rho)\ d^3 x\\
& &\\
& &+\frac{\lambda}{2} \int\int (a^*(x) a(x)-\rho)\phi(x-y)(a^*(y) a(y)-\rho)\ d^3 x d^3 y,
\end{array}
\end{equation} where reference to the integration domain $\Lambda$ is omitted and where the operators $a^*(x)$ and $a(x)$ are the usual creation- and annihilation operators on $\mathcal{F}_{G}.$ They are operator-valued distributions on $\mathcal{F}_{G}$ satisfying the canonical commutation relations
\begin{equation}\label{eq:commutator}
[a^{\#}(x), a^{\#}(y)]=0,\ \  \ [a(x), a^*(y)]=\delta(x-y)I,
\end{equation} $a^{\#} :=a $ or $a^*$.

We are interested in analyzing the dynamics of the system described by Equations ~\eqref{eq:Hamiltonian1}, ~\eqref{eq:Hamiltonian} and ~\eqref{eq:commutator} in the thermodynamic limit, $\Lambda \nearrow \mathbb{R}^{3}.$ In particular, we would like to determine the states of lowest total energy, i.e., the ground states of $H$. This poses a formidable mathematical problem that is far from being understood rigorously. In order to simplify matters, we propose to study the following limiting regime (``mean-field limit''): Let $N>0$ be a parameter that will eventually be taken to $\infty$. We set
\begin{equation}\label{eq:ParaRescaling}
\rho=N\frac{\rho_0}{g^2},\ \lambda=N^{-1} \lambda_0 g^2,\ M=NM_0,\ V(X)=NV_0(X)
\end{equation} where $\rho_0$, $g$, $\lambda_0$, $M_0$ and $V_0(X)$ are fixed (i.e., $N$-independent). For physical motivation underlying the choice ~\eqref{eq:ParaRescaling}, see ~\cite{FSSZ}. We set
\begin{equation}
b^{\#}_{N}(x):=N^{-\frac{1}{2}} a^{\#}(x)-\sqrt{\frac{\rho_0}{g^2}}.
\end{equation} The operators $b_{N}^{*}(x)$ and $b_{N}(x)$ are creation- and annihilation operators satisfying the commutation relations
\begin{equation}\label{eq:NewCommutator}
[b_{N}^{\#}(x), b_{N}^{\#}(y)]=0,\ [b_{N}(x), b_{N}^{*}(y)]=N^{-1} \delta(x-y)I.
\end{equation}

In the new variables the Hamiltonian $H$ introduced in ~\eqref{eq:Hamiltonian} is given by
\begin{equation}\label{eq:HamiRescaled}
H=N H_{N}
\end{equation} where
\begin{equation}\label{dif:HamiHn}
\begin{array}{lll}
H_{N}&:=&-\frac{1}{2N^2 M_0}\Delta_{X}+V_0(X) +\int \frac{1}{2m} (\nabla b_{N}^{*})(x)(\nabla b_{N})(x)\ d^3 x\\
& &\\
& &+g \int W(x-X) \{ b_{N}^{*}(x) b_{N}(x)+\sqrt{\frac{\rho_0}{g^2}} (b_{N}^{*}(x)+b_{N}^{*}(x))\}\ d^3 x\\
& &+\frac{\lambda_0 g^2}{2} \int\int [b_{N}^{*}(x) b_{N}(x)+\sqrt{\frac{\rho_0}{g^2}} (b_{N}^{*}(x)+b_{N}^{*}(x))]\phi(x-y) \ \times\\
& & [b_{N}^{*}(y) b_{N}(y)+\sqrt{\frac{\rho_0}{g^2}} (b_{N}^{*}(y)+b_{N}^{*}(y))] \ d^3 x d^3 y.
\end{array}
\end{equation}
The Schr\"odinger equation for the time evolution of states, $\Psi_{t},$ in the Hilbert space $\mathcal{H}$ introduced in ~\eqref{eq:Hamiltonian1} has the form
\begin{equation}\label{eq:LinSchrodinger}
i\frac{\partial}{\partial t}\Psi_{t} =H\Psi_{t}\ \Longleftrightarrow \ iN^{-1}\frac{\partial}{\partial t}\Psi_{t}=H_{N}\Psi_{t}.
\end{equation} Considering Equations ~\eqref{eq:ParaRescaling}, ~\eqref{eq:NewCommutator}, ~\eqref{dif:HamiHn} and ~\eqref{eq:LinSchrodinger}, we see that $N^{-1}$ plays the role of Planck's constant $\hbar$ and that the mean-field limit, $N\rightarrow \infty,$ is equivalent to a {\it{classical limit}}, ($\hbar \searrow 0$). From Equations ~\eqref{eq:NewCommutator} through ~\eqref{eq:LinSchrodinger} it is easy to heuristically derive the classical Hamiltonian dynamics emerging in the mean-field limit.

The phase space of the limiting classical system is given by
\begin{equation}
\Gamma :=\mathbb{R}^{6} \times \mathcal{H}^{1}(\mathbb{R}^3),
\end{equation} where $\mathcal{H}^{1}(\mathbb{R}^3)$ is complex Sobolev space over $\mathbb{R}^3.$

We choose the usual position- and momentum coordinates $(X,P)$ on $\mathbb{R}^{6}$ and complex coordinate functions $(\beta(x),\ \bar{\beta}(x))$ on $\mathcal{H}^{1}(\mathbb{R}^3)$. The standard symplectic structure on $\Gamma$ yields the Poisson brackets
\begin{equation}
\{X^i,\ X^j\}=\{P_i,\ P_j\}=0,\ \{X^i, \ P_{j}\}=\delta^{i}_{\ j},
\end{equation} $i,j,=1,2,3,$ and
\begin{equation}
\{\beta(x),\ \beta(y)\}=\{\bar{\beta}(x),\ \bar{\beta}(y)\}=0,\ \{\beta(x),\ \bar{\beta}(y)\}=i \delta(x-y).
\end{equation} Corresponding to the Hamiltonian $H_{N}$ we consider the Hamilton functional
\begin{equation}
\begin{array}{lll}
\mathcal{H}&=&\mathcal{H}(X,P;\beta,\bar\beta)\\
&:=&\frac{P^2}{2M_0} +V_0(X)+\int \frac{1}{2m} |\nabla \beta (x)|^2\ d^3x\\
& &+g\int W(x-X) \{|\beta(x)|^2 +2\sqrt{\frac{\rho_0}{g^2}}Re \beta(x)\}\ d^3 x\\
& &+\frac{\lambda_0 g^2}{2} \int\int [|\beta(x)|^2+2\sqrt{\frac{\rho_0}{g^2}} Re\beta(x)]\phi(x-y)[|\beta(y)|^2+2\sqrt{\frac{\rho_0}{g^2}} Re\beta(y)]\ d^3 x d^3 y.
\end{array}
\end{equation}

Setting $\alpha(x)=\beta(x)+\sqrt{\frac{\rho_0}{g^2}},$ the Hamilton functional $\mathcal{H}$ is seen to be given by
\begin{equation}\label{eq:HamiAlpha}
\begin{array}{lll}
\mathcal{H}&=&\frac{P^2}{2M_0}+V_{0}(X)+\frac{1}{2m}\int |\nabla \alpha(x)|^2\ d^3 x\\
&+&g \int W(x-X) (|\alpha(x)|^2-\frac{\rho_0}{g^2}) \ d^3 x\\
&+& \frac{\lambda_0 g^2}{2}\int\int (|\alpha(x)|^2-\frac{\rho_0}{g^2}) \phi(x-y)(|\alpha(y)|^2-\frac{\rho_0}{g^2})\ d^3 xd^3 y,
\end{array}
\end{equation} and the equations of motion are found to be
\begin{equation}\label{eq:XtPt}
\dot{X}_{t}=\frac{P_{t}}{M},\ \
\dot{P}_{t}=-\nabla V(X_{t}) +g \int (\nabla W)(x-X_t) (|\alpha_t (x)|^2-\frac{\rho_0}{g^2})
\end{equation}
and
\begin{equation}\label{eq:Alphat}
i\dot\alpha_{t}(x)=(-\frac{1}{2m}\Delta+g W(x-X_{t}))\alpha_{t}(x)+\lambda_0 g^2 (\phi* (|\alpha_t (x)|^2-\frac{\rho_0}{g^2}))(x)\alpha_{t}(x),
\end{equation} where $*$ indicates a convolution. Setting $W\equiv 0,$ we find the explicit ground state solutions
\begin{equation}\label{eq:ZeroSol}
\begin{array}{lll}
\dot{P}_{t}&=&0,\ X_{t}=X_*: \ \text{a minimum of}\ V_{0}(X),\\
\alpha_{t}&=& \alpha_* = e^{i\theta} \sqrt{\frac{\rho_0}{g^2}},
\end{array}
\end{equation} where $\theta$ is an arbitrary phase. These solutions show that, for $\rho_0>0,$ the continuous gauge symmetry
$$\alpha(x)\rightarrow e^{i\theta} \alpha(x),\ \bar\alpha(x)\rightarrow e^{-i\theta}\bar\alpha(x)$$ of the Hamilton functional ~\eqref{eq:HamiAlpha} is spontaneously broken in the ground states, which corresponds to Bose-Einstein condensation. One then expects that the Bose gas exhibits gapless (Goldstone) modes at zero temperature. In ~\cite{FSSZ}, the ground state solutions are constructed for $W\not=0.$

It has been shown in ~\cite{Hepp} that, in a finite periodic box (torus) $\Lambda$, the dynamics of the quantum system with initial conditions given by a ``coherent state'' in $\mathcal{H}$ is given by evolving the coherent state along a solution of ~\eqref{eq:XtPt}, ~\eqref{eq:Alphat}, with harmonic quantum fluctuations of amplitude $O(\frac{1}{\sqrt{N}})$ around the classical solution, as $N\rightarrow \infty.$ The spectrum of these fluctuations is found by linearizing Equations ~\eqref{eq:XtPt}, ~\eqref{eq:Alphat} around a given classical solution. For $W\equiv 0,$ and choosing the classical solutions to be given by a ground state ~\eqref{eq:ZeroSol}, the frequency spectrum of the harmonic quantum fluctuations of the Bose gas can be found by passing to the ``Bogoliubov limit'' $g\rightarrow 0$. One then finds the dispersion law
\begin{equation}
\omega(k)=|k| \sqrt{\frac{k^2}{4m^2}+\frac{\lambda_0 \rho_0}{m}\hat{\phi}(k)}\displaystyle\approx v_{*}|k|, \ \text{as}\ |k|\rightarrow 0,
\end{equation} where $k\in \Lambda^*$ is the wave vector of a normal mode and $v_*=\sqrt{\frac{\lambda_0\rho_0}{m}\hat{\phi}(0)}$ is the speed of sound in the Bose gas; see \cite{B, FSSZ}, and \cite{S} for a somewhat different approach.

It is a challenging open problem to find out whether the mean-field limit $N\rightarrow \infty$ and the thermodynamic limit $\Lambda \nearrow \mathbb{R}^3$ can be interchanged. (It is not surprising that this problem is difficult, because there is no mathematically rigorous understanding of Bose-Einstein condensation in an interacting Bose gas, in the thermodynamic limit.) However, for an ideal Bose gas $(\lambda_0=0)$ and in the Bogoliubov limit ($g=0$), this problem is understood, ~\cite{Hepp}.

From now on, we suppose that we first pass to the mean-field limit and then to the thermodynamic limit or that we consider models in the Bogoliubov limit.

Our analysis shows that, for vanishing potentials $W$ and $V_0$, the energy-momentum spectrum of the linear quantum fluctuations of the system is as indicated in Figure ~\ref{Figure1}.

\begin{center}\label{Figure1}
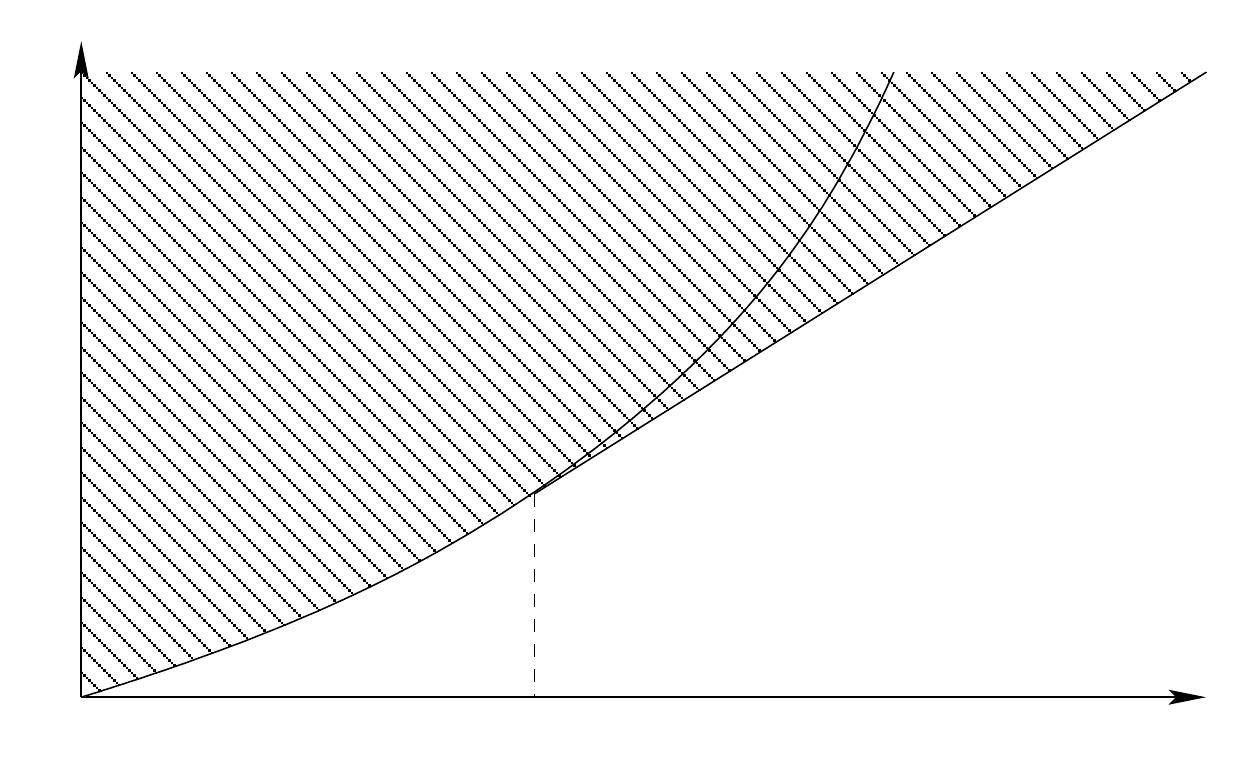
\end{center}

We observe that, for $|P|> P_{*}=Mv_{*}$, the energy, $\frac{P^{2}}{2M}$, of a tracer particle with momentum $P$ is embedded in the continuous energy spectrum. Standard ideas of resonance theory suggest that if the interactions between the tracer particle and the atoms in the Bose gas are turned on, i.e., for $W\not=0,$ a state of the tracer particle corresponding to steady motion at a velocity $v=\frac{P}{M},$ with $|v|\geq v_{*}$, is unstable and decays into states of smaller velocity by emission of gapless modes into the Bose gas (Cerenkov radiation); i.e., the particle motion undergoes deceleration until the speed of the particle has dropped to a value below $v_{*}$, see ~\cite{KM,FSSZ}. Spectral results that go in the direction of confirming this picture have been established in \cite{DFP}. To summarize, we note that, apparently, the motion of a tracer particle interacting with the atoms of a condensed Bose gas exhibits {\it{friction by emission of Cerenkov radiation}} until the speed of the particle is below the speed of sound of the Bose gas.

In this paper, we present mathematical results confirming this picture for the simplest models corresponding to $\lambda_0=0,$ i.e., for an ideal Bose gas, in the mean-field limit. In these models, the speed of sound vanishes, i.e. $v_*=0,$ and the friction mechanism is at work until the particle comes to rest. In the variables $(X, P,\beta,\bar\beta)$, the equations of motion are given by
\begin{equation*}
\begin{array}{lll}
\dot{X}_{t}=\frac{P_{t}}{M_0},\\
\dot{P}_{t}=-\nabla V_0(X)+g\int (\nabla W^{X_{t}})(x) (|\beta_{t}(x)|^2+2\sqrt{\frac{\rho_0}{g^2}} Re\beta_{t}(x))\ d^3 x
\end{array}
\end{equation*} and
\begin{equation}\label{eq:Bhxt}
i\dot\beta_{t}(x)=(h^{X_{t}})\beta_{t}(x)+\sqrt{\rho_0} W^{X_{t}}(x),
\end{equation} where $W^{X}(x):=W(x-X)$ and
\begin{equation}\label{eq:hxt}
h^{X}:=-\frac{1}{2m}\Delta +gW^{X}.
\end{equation}

In order to simplify our notations, we set $m=1$ and $\sqrt{\rho_0}=:\kappa$. We assume that $V_0$ is smooth and $W$ is smooth, spherically symmetric, of rapid decay at infinity, and such that $\int W(x)\ d^3x=\hat{W}(0)\not=0.$ Moreover, for reasons of stability of the system, we choose $g$ to be so small that $h^{X}$ does not have any bound states and zero-energy resonances.

It is reasonable to expect that
\begin{equation}\label{eq:trajectoryStop}
P_{t}\rightarrow 0,\ X_{t}\rightarrow X_{\infty},
\end{equation} where $X_{\infty}$ is a (local) minimum of $V_0,$ and
\begin{equation}\label{eq:BoseGas}
\dot\beta_{t}\rightarrow 0, \ i.e., \ \beta_{t}\rightarrow -\kappa (h^{X_{\infty}})^{-1} W^{X_{\infty}},
\end{equation} as time $t$ tends to $\infty.$ We are interested in establishing ~\eqref{eq:trajectoryStop} and ~\eqref{eq:BoseGas} and to find out how $|P_{t}|$ tends to $0$, as $t\rightarrow 0$, for a physically reasonable class of initial conditions and sufficiently small values of $g.$

In order to develop some intuition into the behavior of $|P_{t}|$ for large times $t$, we consider a steady motion of the tracer particle under the influence of a constant external force, i.e. for $V_{0}(x)=-F\cdot X$, in the limit where $g\rightarrow 0;$ (the ``B-model" in the classification of ~\cite{FSSZ}). Such a motion is a solution of the following equations:
\begin{equation}\label{eq:MotionWithF}
\dot{X}_{t}=v,\ 0=\dot{P}_{t}=F+2\kappa \int (\nabla W^{X_{0}+vt})(x)Re\beta_{t}(x)\ d^3x,
\end{equation} where $(X_0, P_0=M_0 v)$ are the initial conditions for the tracer particle, and
\begin{equation*}
\beta_{t}(x)=\gamma(x-X_0-vt),
\end{equation*} where $\gamma$ solves the equation
\begin{equation}\label{eq:difGamma}
-iv\cdot (\nabla \gamma)(x)=-\frac{1}{2}(\Delta \gamma)(x)+\kappa W(x).
\end{equation} The Fourier transformation, $\hat{\gamma}(k)$, of the solution of ~\eqref{eq:difGamma} is given by
\begin{equation}\label{eq:FourierGamma}
\hat{\gamma}(k)=-2\kappa \frac{\hat{W}(k)}{(k-v)^2-v^2-i0}.
\end{equation} Note that $\hat\gamma$ is singular on the sphere $\{k||k-v|=|v|\}.$ Changing variables, $p:=k-v,$ and plugging ~\eqref{eq:FourierGamma} into ~\eqref{eq:MotionWithF}, we find that
\begin{equation}\label{eq:SizeF}
\begin{array}{lll}
F&=&-Re 4\kappa \int i(p+v) \frac{|\hat{W}(p+v)|^2}{p^2-v^2-i0}\ d^3 p\\
&=&8\kappa \pi \int (p+v)|\hat{W}(p+v)|^2 \delta(p^2-v^2)\ d^3 p.
\end{array}
\end{equation} ~\eqref{eq:SizeF} can be viewed as an instance of Fermi's Golden Rule; see ~\cite{FSSZ}. If $W$ is of short range and $|F|$ is small, so that $|v|$ will turn out to be small, too, we may replace $|\hat{W}(p+v)|^2$ by $|\hat{W}(0)|^2$ on the right hand side of ~\eqref{eq:SizeF}. Then
\begin{equation}
F\approx (4\pi)^2 \kappa \hat{W}(0) v^2 \hat{v},
\end{equation} where $\hat{v}$ is the unit vector in the direction of $v$. Not surprisingly, it follows that $F$ and $v$ are parallel.

Let us now assume that $F$ is taken to $0$, but that, at large times when $|\dot{P}_{t}|$ is very small, with $\frac{P_{t}}{M}\approx v= \text{constant},$ the state of the Bose gas is close to $\beta_{t}(x)=\gamma(x-X_0-vt),$ where $\gamma$ is the solution of ~\eqref{eq:difGamma} just constructed. Then the motion of the tracer particle can be found by solving the equations of motion
\begin{equation}
M_{0}\dot{v}=-(4\pi)^2 \kappa \hat{W}(0)v^2\hat{v}.
\end{equation} Choosing initial conditions $P_0=M_0 u_0 \vec{e}_{z}$ (where $\vec{e}_{z}$ is the unit vector in the $z$-direction), we find that $P_{t}=M_0 u_{t}\vec{e}_{z},$ where
\begin{equation*}
\dot{u}_{t}=-C u_{t}^2, \ (C=(4\pi)^2 \kappa M_{0}^{-1} \hat{W}(0))
\end{equation*} so that
\begin{equation*}
u_{t}=(Ct+\frac{1}{u_0})^{-1}.
\end{equation*} Thus
\begin{equation}
|P_{t}|=O(\frac{1}{t}),\ \text{as}\ t\rightarrow \infty,
\end{equation} and it follows that
\begin{equation}
|X_{t}|=\int_{0}^{t} u_t \ dt\rightarrow \infty,
\end{equation} as $ t\rightarrow \infty.$

This heuristic analysis also shows that if $|\hat{W}(k)|=O(|k|^{-\epsilon}),$ $\epsilon>0,$ or for a particle in a confining external potential $V_0$ with non-degenerate minima, or for a {\it{two-dimensional}} system with $\hat{W}(0)\not=0$, but $V_0\equiv 0,$ $X_{t}$ is predicted to converge to a finite point $X_{\infty},$ as $t\rightarrow \infty.$

The problem with these arguments is that the initial condition of the Bose gas, $\beta_0(x)=\gamma(x-X_0)$, with $\gamma$ as in ~\eqref{eq:difGamma}, ~\eqref{eq:FourierGamma}, is highly singular and has infinite energy, and this is unnatural, physically. We propose to analyze the motion of the particle with initial conditions given by
\begin{equation*}
(X_{t=0},\ P_{t=0})=(X_0, M_0 v), \ |v| \ \text{small},\ \beta_{t=0}=\beta_0
\end{equation*} where $\beta_0$ is small and of reasonably fast decay at infinity. Then the effective equations of motion for the particle exhibit memory effects, and it becomes more subtle to find out how $|P_{t}|$ decays to $0$, as $t\rightarrow \infty.$ This is the problem we wish to solve in this paper for the simplest model with $\lambda_0=0$ and $g\rightarrow 0$. We hope to treat models with $\lambda_0>0$ in the ``Bogoliubov limit", $g\rightarrow 0,$ in future work.
\section{Mathematical Results on Friction in an Ideal Bose Gas, and Strategy of Proofs}
We set the coupling constant $\lambda_0$ to $0$ and assume that the potential $W$
\begin{itemize}
\item[(A1)] is smooth;
\item[(A2)] decays to $0$ exponentially at $\infty;$
\item[(A3)] is spherically symmetric; and
\item[(A4)] has the property that $\hat{W}(0)\not=0.$
\end{itemize}

We state our results for the so-called B-model, $\lambda_0=0,$ $g\rightarrow 0$, $V_0\equiv 0.$ The equations of motion then have the following form (see ~\eqref{eq:Bhxt}, ~\eqref{eq:hxt}):
\begin{equation}\label{eq:Motion}
\begin{array}{lll}
\dot{X}_{t}&=&\frac{P_{t}}{M_0},\ \ \dot{P}_{t}=2\kappa \int (\nabla W^{X_{t}})\ Re\beta_{t}\\
i\dot\beta_{t}&=&-\frac{1}{2}\Delta \beta_{t}+\kappa W^{X_{t}},
\end{array}
\end{equation} with $\kappa=\sqrt{\rho_0}$, $W^{X}(x)=W(x-X)$.

Our main result is the following theorem.
\begin{theorem}
For an arbitrary $\delta\in (0,\delta_*)$, where $\delta_*>0$ is some constant, there exists an $\epsilon=\epsilon(\delta)>0$ such that if
\begin{equation}
\|\langle x\rangle^3 \beta_0\|_{2}\leq \epsilon, \ |P_0|\leq \epsilon,
\end{equation} with $\langle x\rangle=\sqrt{1+|x|^2}$, then
\begin{equation}
|P_{t}|\leq \text{const}\ t^{-\frac{1}{2}-\delta},\ \text{as}\ t\rightarrow \infty,
\end{equation} and
\begin{equation}
\lim_{t\rightarrow \infty}\|\beta_{t}-2\kappa (\Delta)^{-1}W^{X_{t}}\|_{\infty}=0.
\end{equation}
\end{theorem}
\begin{rem*}\
\begin{itemize}
\item[(1)] We do not have a simple explicit formula for the exponent $\delta$, and it is not clear whether it is universal (independent of initial conditions).
\item[(2)] Assumption (A4), which plays a critical role in our analysis, implies that $\Delta^{-1}W^{X}\not\in \mathcal{L}^{2}(\mathbb{R}^{3}, d^3 x)$. In contrast, if $\lambda_0$ is positive we expect that $\beta_{t}$ remains uniformly square-integrable, as $t\rightarrow \infty.$
\end{itemize}
\end{rem*}

Next, we outline the strategy of the proofs of our main results: The equations of motion ~\eqref{eq:Motion} are a system of semilinear integro-differential equations. We solve the equation for $\beta_{\cdot}$, given the particle trajectory $X_{\cdot},$ and plug the result into the equation for $P_{\cdot}.$ We derive, in Subsection ~\ref{subsec:deriveP} below, an effective integro-differential equation for $P_{\cdot}$ of the form
\begin{equation}\label{eq:MotionP}
\dot{P}_{t}=L_{1}(P_{\cdot})(t)+L_{2}(P_{\cdot})(t)+N(P_{\cdot})(t),
\end{equation} where $L_{1}$ and $L_2$ are linear operators from the space $\Upsilon_{[0,t]}:=\{P_{s}\in \mathbb{R}^{3}| 0\leq s\leq t\}$ of momentum trajectories to $\mathbb{R}^{3}$, and $N$ is a non-linear operator from $\Upsilon_{[0,t]}$ to $\mathbb{R}^{3}$. The operator $L_{1}$ is of convolution type, i.e.,
\begin{equation}\label{eq:difF}
L_{1}(P_{\cdot})(t)=-\int_{0}^{t} ds\ f(t-s) P_{s},
\end{equation} where
\begin{equation}\label{eq:difF2}
f(s)=Z Re\langle W, e^{i \frac{\Delta s}{2}}W\rangle,
\end{equation} with $\langle \cdot, \ \cdot\rangle$ the scalar product on $\mathcal{L}^{2}(\mathbb{R}^{3}, d^3x)$ and $Z$ some positive constant. The operator $L_2$ is given by
\begin{equation}\label{eq:difL2p}
L_{2}(P_{\cdot})(t)=f(t)\int_{0}^{t} P_{s}\ ds.
\end{equation}
Let $K_{t}\in \mathbb{R}$, $t\in [0,\infty)$, be the solution of the linear equation
\begin{equation}\label{eq:wienerHopf}
\begin{array}{lll}
\dot{K}_{t}&=&L_{1}(K_{\cdot}) (t) =-\int_{0}^{t} ds\ f(t-s) K_{s},\\
& &\\
K_0&=&1.
\end{array}
\end{equation} By Duhamel's principle, Eq.~\eqref{eq:MotionP} can then be converted into the integral equation
\begin{equation}\label{eq:rewriteK}
P_{t}=K_{t}P_0+\int_{0}^{t} ds\ K_{t-s} L_{2}(P_{\cdot})(s)+\int_{0}^{t} ds\ K_{t-s} N(P_{\cdot})(s).
\end{equation}
It is easy to derive from the properties of the function $f$ that $K_{t}= O(t^{-\frac{1}{2}})$, as $t\rightarrow \infty$ (see ~\eqref{eq:asymK} below). From this we infer that the terms on the right hand side of ~\eqref{eq:rewriteK} do not decay faster than $t^{-\frac{1}{2}},$ as $t\rightarrow \infty.$ This is disappointing, and we propose to show that the leading terms on the right hand side of ~\eqref{eq:rewriteK} cancel each other to yield a decay of $O(t^{-\frac{1}{2}-\delta}),$ as $t\rightarrow \infty,$ for some $\delta>0.$ This cancellation is exhibited as follows: We integrate ~\eqref{eq:MotionP} over time to find that
\begin{equation*}
P_{t}=P_{0}+\int_{0}^{t} ds\ L_{1}(P_{\cdot})(s)+\int_{0}^{t} ds\ L_{2}(P_{\cdot})(s)+\int_{0}^{t} ds \ N(P_{\cdot})(s).
\end{equation*}
Multiplying this equation by $K_{t}$ and subtracting the resulting equation from ~\eqref{eq:rewriteK}, we obtain
\begin{equation}\label{eq:afterSubs}
\begin{array}{lll}
P_{t}(1-K_{t})&=&-K_{t}\int_{0}^{t} ds \ L_{1}(P_{\cdot}) (s)+\int_{0}^{t} ds\ (K_{t-s}-K_{t}) L_2(P_{\cdot})(s) \\
& &\\
& &+\int_{0}^{t} ds\ (K_{t-s}-K_{t}) N(P_{\cdot})(s).
\end{array}
\end{equation}
The first term on the right hand side of ~\eqref{eq:afterSubs} does not have an improved decay in time, yet. We rewrite the second term so as to cancel the leading contribution to the first term. Using ~\eqref{eq:difL2p}, we find
\begin{equation}\label{eq:halfway}
\begin{array}{lll}
& &-\int_{0}^{t}ds\ (K_{t-s}-K_{t})f(s)\int_{0}^{s} du \ P_u\\
& &\\
&=& \int_{0}^{t}ds\ (K_{t-s}-K_{t})f(s)\int_{s}^{t} du \ P_u-\int_{0}^{t} ds\ (K_{t-s}-K_t) f(s)\int_{0}^{t} du\ P_u\\
& &\\
&=& \int_{0}^{t} ds\ (K_{t-s}-K_t)f(s)\int_{s}^{t} du\ P_u-\int_{0}^{t}ds\ K_{t-s} f(s) \int_{0}^{t}du\ P_u\\
& &+K_{t}\int_{0}^{t}ds\ f(s) \int_{0}^{t}du\ P_u.
\end{array}
\end{equation} Using expression ~\eqref{eq:difF2} for $f(s)$, the last term on the right hand side of ~\eqref{eq:halfway} is found to be given by
\begin{equation}
2K_{t}Re\langle W, (i\Delta)^{-1} e^{i\frac{\Delta t}{2}}W\rangle \int_{0}^{t}du\ P_u,
\end{equation} where we have used that $Re\langle W, (i\Delta)^{-1}W\rangle=0.$ The first term on the right hand side of ~\eqref{eq:afterSubs} is rewritten as follows
\begin{equation}\label{eq:rewrite2}
\begin{array}{lll}
-K_{t}\int_{0}^{t}ds\ L_{1}(P_{\cdot})(s)&=&-K_{t}\int_{0}^{t} ds\ \int_{0}^{s} du\ f(s-u) P_{u}\\
& &\\
&=&-K_{t} \int_{0}^{t}ds\ \int_{0}^{s} du\ Re\langle W, e^{i\frac{\Delta(s-u)}{2}} W\rangle P_{u}\\
& &\\
&=& -2K_{t} Re\langle W, (i\Delta)^{-1} \int_{0}^{t} du\ e^{i\frac{\Delta (t-u)}{2}}W\rangle P_{u},
\end{array}
\end{equation} where we have used integration by parts and $Re\langle W, (i\Delta)^{-1}W\rangle=0.$ Combining ~\eqref{eq:halfway} and ~\eqref{eq:rewrite2} we find that
\begin{equation}\label{eq:finalForm}
\begin{array}{lll}
P_{t}(1-K_{t})&=& 2Z K_{t} Re\langle W, (i\Delta)^{-1} \int_{0}^{t} ds\ [e^{i\frac{\Delta (t-s)}{2}}-e^{i\frac{\Delta t}{2}}]W\rangle P_{s}\\
& &\\
& &+Z\int_{0}^{t} ds\ (K_{t-s}-K_{t})Re\langle W, e^{i\frac{\Delta s}{2}}W\rangle \int_{s}^{t}du\ P_{u}\\
& &\\
& &-Z\int_{0}^{t}ds\ K_{t-s} Re\langle W, e^{i\frac{\Delta s}{2}}W\rangle \int_{0}^{t} du\ P_{u}\\
& &\\
& &+\int_{0}^{t} ds \ (K_{t-s}-K_{t}) N(P_{\cdot})(s).
\end{array}
\end{equation} Among the terms on the right hand side of ~\eqref{eq:finalForm}, the third term looks troublesome. In order to show that it has the desired decay, $O(t^{-\frac{1}{2}-\delta})$, with $\delta>0$, we recall that
\begin{equation}\label{eq:higherOrder}
\begin{array}{lll}
Z\int_{0}^{t} ds\ K_{t-s} Re\langle W, e^{i\frac{\Delta s}{2}}W\rangle=\int_{0}^{t} ds\ K_{t-s} f(s)
=L_{1}(K_{\cdot})(t)=\dot{K}_{t}.
\end{array}
\end{equation}
Since $K_{t}=O(t^{-\frac{1}{2}})$, we may expect that $\dot{K}_{t}=O(t^{-\frac{3}{2}}),$ and this is indeed the case. Making the self-consistent assumption that $P_{t}=O(t^{-\frac{1}{2}-\delta}),$ for some $\delta>0$, we find that all the terms on the right hand side of ~\eqref{eq:finalForm} have the appropriate decay in $t.$ In order to extract the decay of $P_{t}$ in $t$ from ~\eqref{eq:finalForm} in a rigorous manner, we would like to convert ~\eqref{eq:finalForm} into an integral equation for $P_{t}$ by dividing both sides by $1-K_{t}$. Since $K_{t}\rightarrow 1$, as $t\rightarrow 0,$ some care is needed for small values of $t$. Since we know that $K_{t}\rightarrow 0$, as $t\rightarrow \infty,$ it suffices to ``wait long enough'' before dividing by $1-K_{t}$. This amounts to dividing the time axis $[0,\infty)$ into two subintervals, $[0,T]$ and $[T,\infty)$, where $T$ is chosen so large that $K_{t}\leq \epsilon$, for $t\geq T$, with $\epsilon$ small enough. On the interval $[0,T]$, we use standard existence and uniqueness theorems to solve Eq.~\eqref{eq:rewriteK}, and, on the interval $[T,\infty)$, Eq.~\eqref{eq:afterSubs} is rewritten appropriately with an inhomogenous part depending on $\{P_{s}\}_{0\leq s\leq T}$. The resulting equation is interpreted as an equation on a Banach space
$$
\mathcal{B}_{\delta,T}:=\{P_{\cdot}|\ |t^{\frac{1}{2}+\delta}||P_{t}|\ \text{uniformly bounded for}\ t\in [T,\infty)\}
$$
with the norm
$$\|P_{\cdot}\|_{\delta,T}:=\|\chi_{[T,\infty)} t^{\frac{1}{2}+\delta}P\|_{\infty}$$ where $\chi_{I}$ denotes the characteristic function of the interval $I\subset\mathbb{R}.$
The idea is then to show that, for an appropriate choice of $\delta>0,$ the equation can be solved in $\mathcal{B}_{\delta,T}$ by a standard fixed-point theorem.

We will discuss some key elements in applying the fixed-point theorem in Subsection ~\ref{subsec:FixedPointTHM}.

\section{Some Technical Details of the Proof}
In this section we explain two elements of our proof: one is to derive Eq.~\eqref{eq:MotionP}, the second one is to show how to make a fixed-point theorem applicable. For a more detailed presentation we refer to a forthcoming paper.
\subsection{Derivation of ~\eqref{eq:MotionP}}\label{subsec:deriveP}
We begin by recasting equations \eqref{eq:Motion} in a more convenient form. We define a function $\delta_t$ by
\begin{align*}
\beta_t=:-2\kappa (-\Delta)^{-1}W^{X_t}+\delta_t\,.
\end{align*}
The equation for $\beta_t$ implies one for $\delta_t$, and using Duhamel's principle, we find that
\begin{align*}
\delta_t=\e^{\i\frac{\Delta t}{2}}\beta_0+2\kappa \e^{\i\frac{\Delta t}{2}}(-\Delta)^{-1}W^{X_0}-\frac{2\kappa}{M_0}\int_0^t ds\ \e^{\i\frac{\Delta(t-s)}{2}}(-\Delta)^{-1}P_s\cdot\nabla_{x} W^{X_s}\,.
\end{align*}
Plugging this equation into the equation for $\dot{P}_t$ we obtain, after some manipulation,
\begin{equation}\label{eq:primitiveP}
\begin{array}{lll}
\dot{P}_t&=&2\kappa \Re\scalar{\nabla_x W^{X_t}}{\e^{\i\frac{\Delta t}{2}}\beta_0}+4\kappa^2\Re\scalar{\nabla_x W}{\e^{\i\frac{\Delta t}{2} }(-\Delta)^{-1}W^{X_0-X_t}}\\
&-&\frac{4\kappa^2 }{M_0}\Re\scalar{\nabla_x W}{\int_0^t ds\ \e^{\i\frac{\Delta(t-s)}{2}}(-\Delta)^{-1}P_s\cdot\nabla_xW^{X_s-X_t}}\,.
\end{array}
\end{equation}
\textbf{Remark:} Here we have used the spherical symmetry of $W$ to cancel terms of the form $\scalar{\nabla_xW}{S}$, where $S$ is spherically symmetric. The fact that such terms vanish will be used repeatedly.\\

Next, we isolate the terms linear in $P$.
\begin{align*}
&\Re\scalar{\nabla_x W}{\e^{\i\frac{\Delta t}{2}}(-\Delta)^{-1}W^{X_0-X_t}}=\Re\scalar{\nabla_x W}{\e^{\i\frac{\Delta t}{2}}(-\Delta)^{-1}(W^{X_0-X_t}-W)}\\
=&\frac{1}{M_0}\Re\scalar{\nabla_x W}{\e^{\i\frac{\Delta t}{2}}(-\Delta)^{-1}\int_0^t ds\ P_s\cdot\nabla_x W}\\
&+\frac{1}{M_0}\Re\scalar{\nabla_xW}{\e^{\i\frac{\Delta t}{2}}(-\Delta)^{-1}\int_0^tds\ P_s\cdot\nabla_x[W^{X_0-X_s}-W]}\,.
\end{align*}
Proceeding with the other term in ~\eqref{eq:primitiveP} in a similar way, we arrive at the equation
\begin{equation}\label{eq:split}
\begin{array}{lll}
\dot{P}_t&=&-\frac{4\kappa^2}{M_0}\Re\scalar{\nabla_x W}{\int_0^t ds\ \e^{\i\frac{\Delta(t-s)}{2}}(-\Delta)^{-1}P_s\cdot\nabla_xW}\\
& &+\frac{4\kappa^2}{M_0}\Re\scalar{\nabla_xW}{\e^{\i\frac{\Delta t}{2}}(-\Delta)^{-1}\int_0^t ds\ P_s \cdot\nabla_x W}\\
& &+N(P)(t)\,,
\end{array}
\end{equation}
where
\begin{equation}
\begin{array}{lll}
N(P)(t)&:=&2\kappa\Re\scalar{\nabla_xW^{X_t}}{\e^{\i\frac{\Delta t}{2}}\beta_0}\\
& &+\frac{4\kappa^2}{M_0}\Re\scalar{\nabla_xW}{\e^{\i\frac{\Delta t}{2}}(-\Delta)^{-1}\int_0^t\ ds\  P_s\cdot\nabla_x[W^{X_0-X_s}-W]}\\
& &-\frac{4\kappa^2}{M_0}\Re\scalar{\nabla_x W}{\int_0^t ds\ \e^{\i\frac{\Delta(t-s)}{2}}(-\Delta)^{-1}P_s\cdot\nabla_x[W^{X_s-X_t}-W]}\\
&=&2\kappa\Re\scalar{\nabla_xW^{X_t}}{\e^{\i\frac{\Delta t}{2}}\beta_0}\\
& &+\frac{4\kappa^2}{M_0^2}\Re\scalar{\nabla_xW}{\e^{\i\frac{\Delta t}{2}}(-\Delta)^{-1}\int_0^t ds\ P_s\cdot\nabla_x\int_{0}^{s} ds_1\ P_{s_1}\cdot \nabla_{x}W^{X_0-X_{s_1}}}\\
& &+\frac{4\kappa^2}{M_0^2}\Re\scalar{\nabla_x W}{\int_0^t ds\ \e^{\i\frac{\Delta(t-s)}{2}}(-\Delta)^{-1}P_s\cdot\nabla_x\int_{s}^{t} ds_1\ P_{s_1}\cdot \nabla_x W^{X_{s_1}-X_{t}}}.
\end{array}
\end{equation}
Using that $W$ is spherically symmetric, we have, for $k=1,2,3,$
\begin{align}
\scalar{\partial_{x_k}W}{\e^{\i\frac{\Delta t}{2}}(-\Delta)^{-1}P_s\cdot\nabla_x W}=\frac{1}{3}P_s^{(k)}\scalar{W}{\e^{\i\frac{\Delta t}{2}}W}\,,
\end{align} where $P^{(k)}$ is the $k$th component of $P,$
and (\ref{eq:split}) is thus found to have the form ~\eqref{eq:MotionP}.

\subsection{A Fixed Point Theorem}\label{subsec:FixedPointTHM}
To render the analysis after Eq.~\eqref{eq:finalForm} rigorous, we reformulate ~\eqref{eq:finalForm} in the following form:
\begin{equation}
\begin{array}{lll}
P_{t}&=&\frac{1}{1-K_t}L(P_{\cdot})(t)+\frac{1}{1-K_t}\int_{0}^{t} ds \ (K_{t-s}-K_{t}) N(P_{\cdot})(s)\\
& &\\
&=&\frac{1}{1-K_t}L(\chi_{[T,\infty)}P_{\cdot})(t)+\frac{1}{1-K_t}\int_{0}^{t} ds \ (K_{t-s}-K_{t}) [N(P_{\cdot})(s)-N(\chi_{[0,T)}P_{\cdot})(s)]\\
& &\\
&+&\frac{1}{1-K_t}L(\chi_{[0,T)}P_{\cdot})(t)+\frac{1}{1-K_t}\int_{0}^{t} ds \ (K_{t-s}-K_{t}) N(\chi_{[0,T)}P_{\cdot})(s)\\
& &\\
&=&\tilde{A}(\chi_{[T,\infty)}P_{\cdot})(t)+\tilde{F}(\chi_{[0,T)}P_{\cdot})(t)
\end{array}
\end{equation} where $L$ is the linear part on the right hand side of ~\eqref{eq:finalForm}, and the terms on the last line are defined to be the first and the second lines after the second equality sign.

To meet the criteria of applicability of a standard fixed-point theorem, we need to establish the following results: (1) The term $\tilde{F}(\chi_{[0,T)}P_{\cdot})$ is small in the Banach space $\mathcal{B}_{\delta,T},$ for some $\delta>0$ and $T>0.$ This will be shown in ~\eqref{eq:smallness}, below. (2) For some $\delta\in (0,\frac{1}{2}),$ the map $\tilde{A}(\chi_{[T,\infty)}\cdot): \ \mathcal{B}_{\delta,T}\rightarrow \mathcal{B}_{\delta,T}$ is a contraction if restricted to a sufficiently small neighborhood of $0.$ This is verified in ~\eqref{eq:contraction}, below. These two facts, together with the observation that $\tilde{A}(\chi_{[T,\infty)}0)\equiv 0,$ obviously make a standard fixed-point theorem applicable and hence imply global existence of a solution in the space $\mathcal{B}_{\delta,T}$.

In what follows we sketch proofs of (1) and (2).

In order to see that $\tilde{F}(\chi_{[0,T)P_{\cdot}})$ is small, for a suitable choice of initial conditions, we observe that it is defined in terms of $\chi_{[0,T)}P_{\cdot}$ and the initial condition $\beta_0$; moreover if $P_0=0$ and $\beta_0\equiv 0$ then $P_{\cdot}\equiv 0$, which implies that $\tilde{F}\equiv 0.$ From a simple proof of local existence of solutions we infer the following proposition.
\begin{proposition}
For any $T>0,$ there exists some $\epsilon=\epsilon(T)>0$ such that
if $\|\langle x\rangle^{-3}\beta_0\|_{2}$, $|P_0|\leq \epsilon $ then
\begin{equation}\label{eq:smallSol}
|P_{t}|\leq T^{-2},\ \text{for any}\ t\leq T.
\end{equation} Moreover, for any $\delta \leq \frac{3}{4}$, there is an $\tilde\epsilon(T)>0$ satisfying $\displaystyle\lim_{T\rightarrow \infty}\tilde\epsilon(T)=0$ such that
\begin{equation}\label{eq:smallness}
\|\tilde{F}(\chi_{[0,T)}P_{\cdot})\|_{\mathcal{B}_{\delta,T}}\leq \tilde\epsilon(T) .
\end{equation}
\end{proposition}

Now we turn to our sketch of the proof of (2). Define a function $\Omega:\ [0,\frac{1}{2})\rightarrow \mathbb{R}^{+}$ by
\begin{equation}\label{eq:difOmega}
\Omega(\delta):=\frac{1}{\pi}\int_{0}^{1}dr\ \frac{1}{1+(1-r)^{\frac{1}{2}}}(1-r)^{-\frac{1}{2}} [\frac{1}{1-2\delta}(r^{-\frac{1}{2}}-r^{-\delta})+ r^{\frac{1}{2}-\delta}].
\end{equation} It is easy to find, by direct computation, that the set $\{\delta| \delta>0,\ \Omega(\delta)< 1\}$ is not empty. One of our key results is the following theorem.
\begin{theorem}
If we restrict the domain of definition of $\tilde{A}(\chi_{[T,\infty)}\cdot)$ to a ball in the space $\mathcal{B}_{\delta,T}$ around $P_{\cdot}\equiv 0$ of sufficiently small radius $\epsilon_0$, for some $\delta>0,$ then $\tilde{A}(\chi_{[T,\infty)}\cdot)$ is a contractive map. Specifically if $\|P_{\cdot}\|_{\mathcal{B}_{\delta,T}}$, $\|Q_{\cdot}\|_{\mathcal{B}_{\delta,T}}\leq \epsilon_0$ then
\begin{equation}\label{eq:contraction}
\|\tilde{A}(\chi_{[T,\infty)}P_{\cdot})-\tilde{A}(\chi_{[T,\infty)}Q_{\cdot})\|_{\mathcal{B}_{\delta},T}\leq [\Omega(\delta)+\epsilon]\|P_{\cdot}-Q_{\cdot}\|_{\mathcal{B}_{\delta},T}
\end{equation} where $\epsilon=\epsilon(|P_0|, \|\langle x\rangle^{-3}\beta_0\|_{2},T)$ has the property that $\epsilon(0,0,\infty)=0.$
\end{theorem}

In what follows we show how to derive ~\eqref{eq:contraction}. We focus our attention on two terms on the right hand side of ~\eqref{eq:finalForm}, namely
\begin{align*}
\Gamma_1(\chi_{[T,\infty)}q_{\cdot} )(t):=-Z\int_0^t ds\ [K(t-s)-K(s)]\Re\scalar{W}{\e^{\i\frac{\Delta s}{2}}W}\int_s^t ds_1\ q_{s_1}\chi_{[T,\infty)}\,,
\end{align*} and
\begin{align*}
\Gamma_2(\chi_{[T,\infty)}q_{\cdot})(t):=2ZK_{t}\ Re\langle W, (i\Delta)^{-1} \int_{0}^{t} ds\ [e^{i\frac{\Delta (t-s)}{2}}-e^{i\frac{\Delta t}{2}}]W\rangle \chi_{[T,\infty)}q_{s},
\end{align*}
where we write $q$ for an arbitrary component of the vector $P$. By direct computation, and using observation ~\eqref{eq:higherOrder}, one can show that all the other terms in $\tilde{A}$ can be bounded by a small constant $\epsilon$, (see ~\eqref{eq:contraction}).

We define two functions $\Omega_1,\ \Omega_2: (0,\frac{1}{2})\rightarrow \mathbb{R}^{+}$ by
\begin{align*}
\Omega_1(\delta):=\frac{1}{(1-2\delta)\pi}\int_0^1 dr\ \frac{1}{1+(1-r)^\pez}(1-r)^\mez [r^\mez-r^{-\delta}] \,.
\end{align*} and
\begin{align*}
\Omega_2(\delta):=\frac{1}{\pi}\int_0^1 dr\ \frac{1}{1+(1-r)^\pez}(1-r)^\mez r^{\pez-\delta}\,.
\end{align*}
Obviously $$\Omega(\delta)=\Omega_1(\delta)+\Omega_2(\delta).$$
A key estimate is given in the following lemma.
\begin{lemma}
There exists some $\epsilon(T)>0$, with $\displaystyle\lim_{T\rightarrow \infty}\epsilon(T)=0$, such that
\begin{align}
\|\Gamma_{k}(\chi_{[T,\infty)}q_{\cdot})\|_{\mathcal{B}_{\delta,T}}\leq [\Omega_k+\epsilon(T)]\|q\|_{\mathcal{B}_{\delta,T}},\ k=1,2.
\end{align}
\end{lemma}
\begin{proof}
We focus on estimating $\Gamma_1$; (the estimate on $\Gamma_2$ can be derived in an almost identical way and hence is omitted).

Using Eq.~\eqref{eq:wienerHopf} and Fourier transformation we find, after some manipulations, that there exists a constant $C_{K}\in \mathbb{R}$ such that
\begin{align}\label{eq:asymK}
ZK(t)=\frac{1}{4|\hat{W}(0)|^2}\pi^{-\frac{5}{2}}t^\mez+C_{K}t^{-1}+O(t^\mdz).
\end{align}

By standard techniques
\begin{align*}
\Re\scalar{W}{\e^{\i\frac{\Delta t}{2}}W}=-2\pi^{\pdz}t^\mdz |\hat{W}(0)|^2+O(t^{-\frac{5}{2}})\,.
\end{align*}
We define $\widetilde{K},\widetilde{M},\widetilde{\Gamma}_1$ to approximate these functions,
\begin{align*}
Z\widetilde{K}(t)&:=\frac{1}{4|\hat{W}(0)|^2}\pi^{-\frac{5}{2}}t^\mez\\
\widetilde{M}&:=-2 |\hat{W}(0)|^2 \pi^{\pdz}t^\mdz\\
\widetilde{\Gamma}_1&:=-Z\int_0^t ds\ [\widetilde{K}(t-s)-\widetilde{K}(s)]\widetilde{M}(s)\int_s^t \ ds_1\ q_{s_1}[1-\chi_T(s_1)]\,.
\end{align*}
We then find that
\begin{align*}
|\widetilde{\Gamma}_1|&\leq \frac{1}{2\pi}\int_0^t ds\ [(t-s)^\mez-t^\mez]s^\mdz\int_s^t ds_1\ |q_{s_1}|\\
&\leq \frac{1}{(1-2\delta)\pi}\int_0^t ds\ [(t-s)^\mez-t^\mez]s^\mdz(t^{\pez-\delta}-s^{\pez-\delta}) \norm{q_{\cdot}}_{\delta,T}\\
&=\frac{1}{(1-2\delta)\pi}\int_0^t ds\ (t-s)^\mez t^\mez\frac{1}{(t-s)^\pez+t^\pez}s^\mez(t^{\pez-\delta}-s^{\pez-\delta})\ \norm{q_{\cdot}}_{\delta,T}\,,
\end{align*}
Changing variables, $s=rt,$ we obtain
\begin{align}
|\widetilde{\Gamma}_1|\leq t^{\mez-\delta}\Omega_1(\delta)\norm{q_{\cdot}}_{\delta,T}.
\end{align}

In the following we estimate the remainder $|\Gamma_1-\widetilde{\Gamma}_1|$. Observe that $\tilde{K}$ and $\tilde{M}$ are good approximations of $K$ and $Re\langle W, e^{i\frac{\Delta t}{2}}W\rangle$ only when $t$ is sufficiently large. Hence to estimate $|\Gamma_1-\widetilde{\Gamma}_1|$ we divide the integration domain $[0,t]$ into three parts, $[0,T^{\frac{1}{3}}]$, $[T^{\frac{1}{3}},t-T^{\frac{1}{3}}]$, and $[t-T^{\frac{1}{3}},t]$. As the estimates on different intervals are very similar, we only consider the first interval.
\begin{align*}
I_1:=&Z\int_0^{T^{\frac{1}{3}}}ds\ [K(t-s)-K(s)]\Re\scalar{W}{\e^{\i\frac{\Delta s}{2}}W}\int_s^t ds_1\ q_{s_1} [1-\chi_T(s_1)]\\
-&Z\int_0^{T^{\frac{1}{3}}}ds\ [\widetilde{K}(t-s)-\widetilde{K}(s)]\widetilde{M}(s)\int_s^t ds_1\ q_{s_1}[1-\chi_T(s_1)]\,.
\end{align*}
By ~\eqref{eq:asymK} we have that
\begin{align*}
|K(t-s)-K(t)|
\lesssim &t^\mez(t-s)^\mez\frac{s}{t^\pez+(t-s^\pez)}+(t-s)^{-1}-t^{-1}+t^{-\frac{3}{2}}\\
\lesssim &t^\mdz(1+s),
\end{align*}
because $s\leq T^{\frac{1}{3}}$, and $t\geq T$. This, together with the fact that $|\widetilde{K}(t-s)-\widetilde{K}(t)|\lesssim t^{-\frac{3}{2}}s,$ implies
\begin{align*}
|K(t-s)-K(t)||\Re\scalar{W}{\e^{\i\frac{\Delta s}{2}}W}|+|\widetilde{K}(t-s)-\widetilde{K}(t)||\widetilde{M}(s)|\lesssim t^\mdz s^\mez\,.
\end{align*}
Plugging this into $I_1$ we obtain that
\begin{align*}
|I_1|\lesssim t^{-1-\delta}\int_0^{T^{\frac{1}{3}}} ds\ s^\mez \norm{q_{\cdot}}_{\delta,T}=t^{-1-\delta}2T^{\frac{1}{6}}\norm{q_{\cdot}}_{\delta,T}\lesssim t^{\mez-\delta}T^{-\frac{1}{3}}\norm{q_{\cdot}}_{\delta,T}\,.
\end{align*}
All together, we conclude that
\begin{align*}
|\Gamma_1|\lesssim t^{\mez-\delta}[\Omega_1(\delta)+\epsilon(T)]\norm{q_{\cdot}}_{\delta,T}\,,
\end{align*}
where $\epsilon(T)\to 0$, as $T\to\infty$.\\

This completes the outline of our proof.
\end{proof}

\end{document}

%% file: test7.pdf_t
\begin{picture}(0,0)%
\includegraphics{test7.pdf}%
\end{picture}%
\setlength{\unitlength}{3947sp}%
\begingroup\makeatletter\ifx\SetFigFontNFSS\undefined%
\gdef\SetFigFontNFSS#1#2#3#4#5{%
  \reset@font\fontsize{#1}{#2pt}%
  \fontfamily{#3}\fontseries{#4}\fontshape{#5}%
  \selectfont}%
\fi\endgroup%
\begin{picture}(5955,3636)(2611,-5464)
\put(8551,-5236){\makebox(0,0)[lb]{\smash{{\SetFigFontNFSS{12}{14.4}{\familydefault}{\mddefault}{\updefault}{\color[rgb]{0,0,0}$|P|$}%
}}}}
\put(2626,-2086){\makebox(0,0)[lb]{\smash{{\SetFigFontNFSS{12}{14.4}{\familydefault}{\mddefault}{\updefault}{\color[rgb]{0,0,0}$E$}%
}}}}
\put(6751,-2011){\makebox(0,0)[lb]{\smash{{\SetFigFontNFSS{12}{14.4}{\familydefault}{\mddefault}{\updefault}{\color[rgb]{0,0,0}$\frac{P^2}{2M}$}%
}}}}
\put(4801,-5386){\makebox(0,0)[lb]{\smash{{\SetFigFontNFSS{12}{14.4}{\familydefault}{\mddefault}{\updefault}{\color[rgb]{0,0,0}$P_{\star}=Mv_{\star}$}%
}}}}
\end{picture}%